\documentclass[11pt]{article}
\usepackage{amsmath}
\usepackage{amsfonts}
\usepackage{amssymb}
\usepackage{amsthm}

\theoremstyle{remark}
\newtheorem{theorem}{Theorem}
\newtheorem{lemma}{Lemma}
\newtheorem{prop}{Proposition}
\newtheorem{corollary}{Corollary}
\newtheorem{example}{Example}
\theoremstyle{definition}
\newtheorem{definition}{Definition}

\newcommand{\hadaprod}{\circ}
\newcommand{\regprod}{*}
\newcommand{\semiprod}{*_\rtimes}
\newcommand{\directprod}{*_\times}
\newcommand{\superclass}[1]{\left[#1\right]}
\newcommand{\aut}{\operatorname{Aut}}
\newcommand{\irr}{\operatorname{Irr}}
\newcommand{\im}{\operatorname{Im}}

\author{Alexander Lang}
\title{Supercharacter Theories and Semidirect Products}

\begin{document}
\maketitle

\begin{abstract}
We describe the supercharacter theories of $H\rtimes K$ in terms of the supercharacter theories of $H\times K$ in the case when both $H$ and $K$ are Abelian. To do this we introduce the concept of a homomorphism of supercharacter theories. This provides a classification of the supercharacter theories of the dihedral groups of order $2m$ when $m$ is odd using the known classification of the supercharacter theories of cyclic groups.
\end{abstract}

\section{Introduction}
The\let\thefootnote\relax\footnote{Research partially supported by NSA grant H98320-12-1-0232 and NSF VIGRE grant DMS-0636297.} notion of a \textit{supercharacter theory} for a finite group was introduced in 2008, by Diaconis and Isaacs \cite{DiaconisIsaacs}. Since then it has become a useful tool in many different areas, including number theory and combinatorial Hopf algebras. It was also discovered in 2010, by Hendrickson that supercharacter theories are a specific kind of Schur ring \cite{Hend}, which allows us to use techniques developed for Schur rings to study supercharacter theories. 

The question of classifying all supercharacter theories for a given class of groups appears difficult. Despite the existence of several general techniques for constructing supercharacter theories, only in the case of cyclic groups have supercharacter theories been classified \cite{cyclicSchur}, \cite{cyclicSchurII}. In particular, the situation of general Abelian groups is not known, because the supercharacter theories of a direct product of groups $H \times K$ are not known in terms of the supercharacter theories of $H$ and $K$. We apply notions from the Schur ring literature to show that in some cases the supercharacter theories of the semidirect product $H\rtimes K$ are completely determined by those of the direct product $H\times K$. In particular, we use the notion of a homomorphism of supercharacter theories to show that when $H$ and $K$ are Abelian every supercharacter theory of $H\rtimes K$ is isomorphic to a supercharacter theory of $H\times K$.

\section{Preliminaries}

We first begin by recalling some of the basic notions of supercharacter theories, for more details see \cite{DiaconisIsaacs}. We denote the identity of a finite group $G$ by $e$ and the character of the trivial representation by $triv$.

\begin{definition}
Given a finite group $G$ with identity $e$, a \textit{supercharacter theory} for $G$ is a pair $(\mathcal{X},\mathcal{K})$ where $\mathcal{X}$ is a partition of $\irr(G)$ the set of irreducible characters of representations of $G$ over $\mathbb{C}$ and $\mathcal{K}$ is a partition of the set of conjugacy classes of $G$ satisfying the following conditions:

(i) $\{e\}$ is an element of $\mathcal{K}$, and $\{triv\}$ is an element of  $\mathcal{X}$

(ii) $|\mathcal{K}|=|\mathcal{X}|$

(iii) for all $X\in \mathcal{X}$ and all $K\in \mathcal{K}$ the class functions $\displaystyle \sigma_X=\sum_{\chi\in X}{\chi(e)\chi}$ satisfy $\sigma_X(g)=\sigma_X(h)$ for all $g,h\in K$.

The elements of $\mathcal{K}$ are called \textit{superclasses} and the $\sigma_X$, for all $X \in \mathcal{X}$ are called \textit{supercharacters}. We shall denote the superclass containing the element $g$ by $\superclass{g}_{\mathcal{K}}$ or by $\superclass{g}$ when $\mathcal{K}$ is clear from context.
\end{definition}

We note that the condition $\{triv\}\in \mathcal{X}$ is redundant, as it is implied by the other conditions in the definition.

There are several ways of constructing supercharacter theories. We recall the following.
\begin{example}
 Letting each conjugacy class be a superclass and $\mathcal{X}$ contain only singletons i.e. letting $\mathcal{K}$ and $\mathcal{X}$ be the finest partitions gives a supercharacter theory referred to as the \textit{minimal} supercharacter theory.
\end{example}

\begin{example}
 Letting $\mathcal{K}=\{\{e\},\{\bigcup_{g\not= e}{g}\}\}$ and $\mathcal{X}=\{\{triv\},\{\bigcup_{\chi \not= triv}{\chi}\}\}$ gives a supercharacter theory known as the \textit{maximal} supercharacter theory. 
\end{example}

The set of all supercharacter theories for $G$ forms a lattice, where the minimal and maximal supercharacter theories are the minimal and maximal elements of the lattice respectively. The partial ordering is given by the partial ordering on the partitions of $G$, or equivalently on the partitions of $\irr(G)$.

\begin{example}
\label{ex:auto}
A subgroup $H$ of $\aut(G)$ acts on both the set of conjugacy classes of $G$ and $\irr(G)$. Letting $\mathcal{K}$ be the set of orbits of the action on the conjugacy classes and $\mathcal{X}$ be the orbits of the action on $\irr(G)$ yields a supercharacter theory. Note that if $H$ acts trivially on the set of conjugacy classes, the resulting supercharacter theory is the minimal supercharacter theory. In particular, every subgroup of the group of inner automorphisms yields the minimal supercharacter theory.
\end{example}

We will utilize an alternative description of a supercharacter theory using the fact that the set of supercharacter theories for a finite group $G$ is in bijective correspondence with the set of commutative Schur rings of $G$ \cite{Hend}. See \cite{Muz1}, \cite{Muz2}, \cite{Kerbythesis}, \cite{Kerby} for a discussion of these conditions in the Schur ring setting.

\begin{definition}
The \textit{Hadamard product} $\hadaprod$ is defined on $\mathbb{C}G$ by 
\begin{equation}
\displaystyle \left(\sum_{g\in G}{a_g g}\right)\hadaprod \left(\sum_{g\in G}{b_g g}\right)=\sum_{g\in G}{(a_g b_g) g}.
\end{equation}
\end{definition}

Note that $(\mathbb{C}G,\hadaprod)$ is a commutative associative algebra, with identity $\displaystyle \sum_{g\in G}{g}$. We shall denote the usual product on $\mathbb{C}G$ by $\regprod$, to distinguish it from the Hadamard product.

\begin{prop}
\label{algebraDef}
For a finite group $G$ there is a bijective correspondence between supercharacter theories $(\mathcal{X},\mathcal{K})$  and $\mathbb{C}$-linear subspaces $A$ of $Z(\mathbb{C}G)$ containing $e$ and $\displaystyle \sum_{g\in G}{g}$ which are closed under the operations $\regprod$ and $\hadaprod$.
\end{prop}

The bijection is as follows. Let $(\mathcal{X},\mathcal{K})$ be a supercharacter theory. For $K\subseteq G$ define $\displaystyle \widehat{K}=\sum_{g \in K}{g}$. Then $\{\widehat{K}: K\in \mathcal{K}\}$ is a $\mathbb{C}$-basis for a linear subspace $A$ which is closed under $\regprod$ and $\hadaprod$. Further, $e\in A$, and $\displaystyle \sum_{g\in G}{g} \in A$. Recall that the set of 
\begin{equation}
e_\chi=\frac{\chi(e)}{|G|}\sum_{g\in G}{\overline{\chi(g)}g}
\end{equation}
 where $\chi \in \irr(G)$ is a $\mathbb{C}$-basis of $Z(\mathbb{C}G)$ consisting of orthogonal primitive central idempotents. For $X\in\mathcal{X}$ let $\displaystyle E_X=\sum_{\chi\in X}{e_\chi}$. Then $\{E_X:X \in \mathcal{X}\}$ is also a $\mathbb{C}$-basis for $A$. Note that $A=Z(\mathbb{C}G)$ when $(\mathcal{X},\mathcal{K})$ is the minimal supercharacter theory.

Conversely, if $A\subseteq Z(\mathbb{C}G)$ is a linear subspace containing $e$ and $\displaystyle \sum_{g\in G}{g}$ and closed under both $\regprod$ and $\hadaprod$, then there is a unique partition $\mathcal{K}$ of the set of conjugacy classes such that $\{\widehat{K}: K\in \mathcal{K}\}$ is a basis for $A$ and a unique partition $\mathcal{X}$ of $\irr(G)$ such that $\{E_X: X \in \mathcal{X}\}$ is a basis of $A$. Then the pair $(\mathcal{X},\mathcal{K})$ is a supercharacter theory. See \cite{Hend} for a discussion and proof of this proposition.

For the rest of our discussion, we shall be primarily interested in algebras satisfying the conditions of Proposition \ref{algebraDef}. Because of the above bijection we shall also refer to such algebras as supercharacter theories. Given such an algebra $A$ we shall denote the corresponding partitions of $\irr(G)$ and $G$ by $\mathcal{X}_A$ and $\mathcal{K}_A$ respectively.

\section{Homomorphisms of Supercharacter Theories}
\begin{definition}

Let $A$ be a supercharacter theory for $G$ and $B$ be a supercharacter theory for $H$. A \textit{homomorphism} of supercharacter theories is a map $\phi: A \rightarrow B$ which is an algebra homomorphism with respect to both $\regprod$ and $\hadaprod$, i.e. $\phi$ is a $\mathbb{C}$-linear map such that for all $x,y\in A$, $\phi(x\regprod y)=\phi(x)\regprod \phi(y)$ and $\phi(x\hadaprod y)=\phi(x)\hadaprod \phi(y)$.
\end{definition}

We note that in the case of an isomorphism of supercharacter theories, this definition is equivalent to the ones given in \cite{Olaf}, and \cite{Kerby}. In \cite{Kerby} we have the following result.

\begin{lemma}
\label{superclassToSuperclass}
Let $A$ be a supercharacter theory for $G$, $B$ a supercharacter theory for $H$. If $\phi:A\rightarrow B$ is a supercharacter theory isomorphism then for every $K \in \mathcal{K}_A$ there exists an $L\in \mathcal{K}_B$ such that $\phi(\widehat{K})=\widehat{L}$.

\end{lemma}
See \cite{Muz1}, \cite{Muz2}, \cite{Kerbythesis}, and \cite{Kerby} for details. By an identical argument as that in \cite{Kerbythesis}, we see that an analogous result holds for the partitions of $\irr(G)$ and $\irr(H)$. For completeness we present a proof.

\begin{lemma}
\label{supercharToSuperchar}
Let $A$ be a supercharacter theory for $G$, $B$ a supercharacter theory for $H$. If $\phi:A\rightarrow B$ is a supercharacter theory isomorphism then for every $X\in \mathcal{X}_A$ there exists a $Y\in \mathcal{X}_B$ such that $\phi(E_X)=E_Y$.
\end{lemma}

\begin{proof}

Let $X\in \mathcal{X}_A$. Since $\{E_Y:Y \in \mathcal{X}_B\}$ is a basis of $B$ we have $\displaystyle \phi(E_X)=\sum_{Y\in \mathcal{X}_B}{a_YE_Y}$ for some $a_Y\in \mathbb{C}$. Since 
\begin{equation}
\phi(E_X)=\phi(E_X\regprod E_X)=\phi(E_X)\regprod \phi(E_X)
\end{equation}
we have 
\begin{equation}
\sum_{Y\in \mathcal{X}_B}{a_Y E_Y}=\sum_{Y\in \mathcal{X}_B}{a_Y^2 E_Y}
\end{equation}
so $a_Y$ is either 0 or 1 for all $Y$. We want to show that there is only one nonzero $a_Y$. We have
\begin{equation}
E_X=\sum_{Y\in \mathcal{X}_B}{a_Y\phi^{-1}(E_Y)}.
\end{equation}
Since $\phi^{-1}$ is also a supercharacter theory isomorphism, there exist $b_{Y,W}\in\{0,1\}$ such that
\begin{equation}
E_X=\sum_{Y\in \mathcal{X}_B}{a_Y\sum_{W\in \mathcal{X}_A}{b_{Y,W}E_W}}.
\end{equation}
Since $\{E_X:X\in \mathcal{X}_A\}$ is linearly independent, $\phi^{-1}$ is injective, and $a_Yb_{Y,W}\in \{0,1\}$ we see that there must be exactly one nonzero $a_Y$. Hence there exists $Y$ such that $\phi(E_X)=E_Y$.
\end{proof}

We will also need the following lemma from \cite{Kerbythesis}, \cite{Kerby}.
\begin{lemma}
\label{sizeOfSuperclasses}
Let $A$ be a supercharacter theory for $G$, $B$ a supercharacter theory for $H$, and $\phi:A\rightarrow B$ a supercharacter theory isomorphism. If $K\in \mathcal{K}_A$ and $L\in \mathcal{K}_B$ such that $\phi(\widehat{K})=\widehat{L}$, then $|K|=|L|$.
\end{lemma}

Note that this implies that $|G|=|H|$.

\section{Supercharacter Theory Isomorphisms and Supercharacter Tables}
We now describe a relationship between supercharacter theory isomorphisms and character tables. A version of this relationship is the \textit{magic rectangle condition} of \cite{Johnson} which does not use the language of supercharacters, and is known to hold more generally, see \cite{Bannai} and \cite{Quasigroupspaper}. For completeness we present a proof for our current situation.

\begin{definition}

A \textit{supercharacter table} of a supercharacter theory is the matrix where the $i,j$th entry is the $i$th supercharacter evaluated at the $j$th superclass for some ordering of the superclasses and supercharacters.
\end{definition}

If $\sigma_X$ is the $i$th supercharacter and $g$ is an element in the $j$th superclass the $i,j$th entry is given by 
\begin{equation}
\displaystyle \sum_{\chi \in X}{\chi(e)\chi(g)}.
\end{equation}
Note that the supercharacter table for the minimal supercharacter theory is obtained by scaling each row of the usual character table by its first entry, i.e. each irreducible character is scaled by its dimension.

\begin{theorem}
Let $A$ be a supercharacter theory for $G$, $B$ a supercharacter theory for $H$. There is an isomorphism of supercharacter theories $\phi:A\rightarrow B$ iff the supercharacter tables of $A$ and $B$ are identical up to reordering of rows and columns.
\end{theorem}
\begin{proof}

Let $\phi:A \rightarrow B$ be an isomorphism of supercharacter theories. Let $X\in \mathcal{X}_A$. Then
\begin{equation}
\phi(E_X)=\frac{1}{|G|}\sum_{\chi \in X}{\sum_{\superclass{g}\in \mathcal{K}_A}{\chi(e)\overline{\chi(g)}\phi(\widehat{\superclass{g}})}}.
\end{equation}
 By Lemma \ref{supercharToSuperchar} there is a $Y\in \mathcal{X}_B$ such that $\phi(E_X)=E_Y$. Then we have
\begin{equation}
\label{eq:x}
\frac{1}{|G|}\sum_{\chi \in X}{\sum_{\superclass{g}\in \mathcal{K}_A}{\chi(e)\overline{\chi(g)}\phi(\widehat{\superclass{g}})}}=\frac{1}{|H|}\sum_{\tilde{\chi}\in Y}{\sum_{\superclass{h}\in \mathcal{K}_B}{\tilde{\chi}(e)\overline{\tilde{\chi}(h)}\widehat{\superclass{h}}}}.
\end{equation}
 For a fixed a superclass $\superclass{h}$ by Lemma \ref{superclassToSuperclass} there is a unique superclass $\superclass{g}$ such that $\phi(\widehat{\superclass{g}})=\widehat{\superclass{h}}$. Equating coefficients in equation \ref{eq:x} gives
\begin{equation}
\frac{1}{|G|}\sum_{\chi \in X}{\chi(e)\overline{\chi(g)}}=\frac{1}{|H|}\sum_{\tilde{\chi}\in Y}{\tilde{\chi}(e)\overline{\tilde{\chi}(h)}}.
\end{equation}
Taking complex conjugates we have
\begin{equation}
\label{eq:y}
\frac{1}{|G|}\sum_{\chi \in X}{\chi(e)\chi(g)}=\frac{1}{|H|}\sum_{\tilde{\chi}\in Y}{\tilde{\chi}(e)\tilde{\chi}(h)}.
\end{equation}
By Lemma \ref{sizeOfSuperclasses} $|G|=|H|$ giving
\begin{equation}
\sum_{\chi\in X}{\chi(e)\chi(g)}=\sum_{\tilde{\chi}\in Y}{\tilde{\chi}(e)\tilde{\chi}(h)}.
\end{equation}
So we see that the supercharacter tables have the same entries.

Conversely, suppose that the supercharacter tables are identical $n$ by $n$ matrices after appropriate rearrangement. Each column corresponds to a superclass, let $K_1,\ldots, K_n$ and $L_1,\ldots, L_n$ be the superclasses corresponding to columns $1,\ldots,n$ after rearranging for $G$ and $H$ respectively. Let $\phi:A\rightarrow B$ be the $\mathbb{C}$-linear map defined by $\phi(\widehat{K_i})=\widehat{L_i}$ for $i=1,\ldots,n$. Then $\phi$ preserves $\hadaprod$. Each row corresponds to a supercharacter, let $X_1,\ldots,X_n$ and $Y_1,\ldots,Y_n$ be the subsets of $\irr(G)$ and $\irr(H)$ respectively corresponding to rows $1,\ldots,n$. Because the tables' entries are identical, $\phi(E_{X_i})=E_{Y_i}$ for all $i$, so $\phi$ is an algebra isomorphism. Hence $\phi$ is a supercharacter theory isomorphism.
\end{proof}

\section{Semidirect Products}

Let $H,K$ be finite Abelian groups and $H \rtimes_{\psi} K$ the semidirect product corresponding to a homomorphism $\psi:K \rightarrow \aut(H)$. We show that there exists a supercharacter theory $A$ for $H\times K$ which is isomorphic to the minimal supercharacter theory for $H \rtimes_{\psi} K$. This implies that every supercharacter theory of $H \rtimes_{\psi} K$ is isomorphic to a supercharacter theory of $H \times K$ and there is a bijection between the supercharacter theories of $H \rtimes_{\psi} K$ and supercharacter theories $(\mathcal{X},\mathcal{K})$ of $H \times K$ which are coarser than $(\mathcal{X}_A,\mathcal{K}_A)$. In particular, the lattice of supercharacter theories for $H \rtimes_{\psi} K$ is a sublattice of the lattice of supercharacter theories for $H\times K$.

\begin{theorem}

Let $H$, $K$ be finite Abelian groups and $H\rtimes_{\psi} K$ be the semidirect product corresponding to $\psi:K\rightarrow \aut(H)$. Then there exists a supercharacter theory $A$ of $H \times K$ which is isomorphic to the minimal supercharacter theory of $H \rtimes_{\psi} K$.
\end{theorem}

\begin{proof}

Let $M$ be the minimal supercharacter theory for $H \rtimes_{\psi} K$. We consider both $H \times K$ and $H \rtimes_{\psi} K$ to be groups whose underlying set is the Cartesian product of $H$ and $K$. Let $\directprod$ be the product in $H\times K$ and $\semiprod$ the product in $H \rtimes_{\psi} K$. We denote $\psi(k)$ by $\psi_k$ and recall that $(h_1,k_1)\semiprod (h_2,k_2)=(h_1 \psi_{k_1}(h_2),k_1k_2)$. 

We proceed by defining an injective linear map $\phi:M\rightarrow \mathbb{C}(H\times K)$, and showing that $\phi$ is a supercharacter homomorphism. We will then show that the desired supercharacter theory of $H\times K$ will be $A=\im(\phi)$. Note that since $H\times K$ is Abelian, any subset of $H\times K$ is a union of conjugacy classes.

 Let $\phi:M\rightarrow \mathbb{C}(H\times K)$ be the $\mathbb{C}$-linear map defined by setting $\phi(\widehat{C})=\widehat{C}$ where $C$ is a subset of the Cartesian product of $H$ and $K$ which is a conjugacy class of $H \rtimes_{\psi} K$ and equivalently a superclass of $M$. Clearly $\phi$ is a homomorphism of $\mathbb{C}$-algebras with respect to $\hadaprod$. We now show that $\regprod$ is preserved. 

To show that $\regprod$ is preserved, let $C_1$, $C_2$ be conjugacy classes in $H \rtimes_{\psi} K$. We want to show that
\begin{equation}
\displaystyle \left( \sum_{(h,k)\in C_1}{(h,k)}\right)\directprod \left( \sum_{(h',k')\in C_2}{(h',k')} \right)
\end{equation}
is equal to
\begin{equation}
\left( \sum_{(h,k)\in C_1}{(h,k)}\right)\semiprod \left( \sum_{(h',k')\in C_2}{(h',k')} \right).
\end{equation}

Fix $(h,k)\in C_1$ and $(h',k') \in C_2$, and let $n$ be the smallest positive integer such that $\psi_{k^n}(h')=h'$. Then 

\begin{equation}
(h,k)\semiprod \left( \sum_{i=1}^n{(e,k^i)\semiprod(h',k')\semiprod(e,k^{-i})}\right)
\end{equation}

\begin{equation}
=(h,k)\semiprod \left( \sum_{i=1}^n{(\psi_{k^i}(h'),k')} \right)
\end{equation}

\begin{equation}
=\sum_{i=1}^n{(h,k)\semiprod (\psi_{k^i}(h'),k')}=\sum_{i=1}^n{(h\psi_{k^{i+1}}(h'),kk')}
\end{equation}

\begin{equation}
=\sum_{i=1}^n{(h\psi_{k^i}(h'),kk')}=(h,k)\directprod \left( \sum_{i=1}^n{(\psi_{k^i}(h'),k')} \right)
\end{equation}

\begin{equation}
=(h,k)\directprod \left( \sum_{i=1}^n{(e,k^i)\semiprod(h',k')\semiprod(e,k^{-i})} \right).
\end{equation}

Since $C_1\times C_2$ is a disjoint union of sets of the form 
\begin{equation}
\bigcup_{i=1}^n \left\{\left( (h,k), (e,k^i)\semiprod(h',k')\semiprod(e,k^{-i}) \right)\right\}
\end{equation}
 where $n$ is as above, we have our desired equality. Therefore $\phi$ preserves $\regprod$, and is a supercharacter homomorphism. Since $e\in \im(\phi)$ and $\sum_{g\in H\times K}{g}\in \im(\phi)$, $A=\im(\phi)$ is a supercharacter theory of $H\times K$. Since $\phi$ is injective $M$ is isomorphic to $A$.
\end{proof}

We note the following property of the supercharacter theory $\im(\phi)$. Since $\im(\psi) \subseteq \aut(H)\subseteq \aut(H\times K)$, it induces a supercharacter theory of $H\times K$ by the method of example \ref{ex:auto}. The supercharacter theory $\im(\phi)$ is always equal to or coarser than the one induced by $\im(\psi)$ because conjugation by elements of the form $(e,k)$ in $H \rtimes_{\psi} K$ satisfies:

\begin{equation}
(e,k)\semiprod (h,k') \semiprod (e,k)^{-1}=(\psi_k(h),k').
\end{equation}

Since the supercharacter theories of cyclic groups have been classified, see \cite{cyclicSchur} and \cite{cyclicSchurII}, the above result provides a classification of the supercharacter theories for some dihedral groups:

\begin{corollary}
If $G$ is a dihedral group of order $2m$ where $m$ is odd, all its supercharacter theories are isomorphic to supercharacter theories for $C_{2m}$ the cyclic group of order $2m$.
\end{corollary}

\bibliographystyle{alpha}
\bibliography{supercharactersSemidirect}

\textsc{Department of Mathematics, University of California, Davis, One Shields Avenue, Davis, CA} 95616-8633

\textit{E-mail address}:\texttt{ langa@math.ucdavis.edu}

\end{document}